\documentclass[11pt]{amsart}
\usepackage{graphicx}
\usepackage{amsmath}
\usepackage{amssymb}
\usepackage{amsfonts}
\usepackage{bm}
\usepackage[all]{xy}
\usepackage{comment}
\usepackage{ascmac}

\makeatother
\newcommand{\C}{\mathbb{C}}
\newcommand{\R}{\mathbb{R}}

\newcommand{\Z}{\mathbb{Z}}
\newcommand{\K}{\mathbb{K}}
\newcommand{\Ha}{\mathbb{H}}

\newcommand{\rank}{\operatorname{rank}}
\newcommand{\mfp}{\frak{p}}
\newcommand{\Herm}{\operatorname{Herm}}

\newcommand{\Sym}{\operatorname{Sym}}
\newcommand{\trace}{\operatorname{trace}}
\newcommand{\Ad}{\operatorname{Ad}}

\newcommand{\diag}{\operatorname{diag}}

\newcommand{\A}{{}^\forall}

\newcommand{\trans}{{}^t\!}
\newcommand{\cb}{\makebox[0pt][l]{$\square$}\raisebox{.15ex}{\hspace{0.1em}$\checkmark$}}
\newcommand{\sq}{$\square$}

\newtheorem{theorem}{Theorem}[section]
\newtheorem{lemma}[theorem]{Lemma}
\newtheorem{corollary}[theorem]{Corollary}
\newtheorem{proposition}[theorem]{Proposition}
\newtheorem{definition}[theorem]{Definition}

\newtheorem{problem}[theorem]{Problem}
\newtheorem{remark}[theorem]{Remark}

\newtheorem{fact}[theorem]{Fact}
\newtheorem{openproblem}[theorem]{Open Problem}

\begin{document}



\title[]{Obstructions to the existence of compact Clifford--Klein forms for tangential symmetric spaces}

\address{Graduate School of Mathematical Science, The University of Tokyo, 3-8-1 Komaba, Meguro-ku, Tokyo 153-8914, Japan,\footnote{The present affiliation is RIKEN Center for Advanced Intelligence Project, Tokyo, Japan, koichi.tojo@riken.jp} }
\email{koichi.tojo@riken.jp}
\author{Koichi Tojo}


\maketitle

\begin{abstract}
For a homogeneous space $G/H$ of reductive type, we consider the tangential homogeneous space $G_\theta/H_\theta$. 
In this paper, we give obstructions to the existence of compact Clifford--Klein forms for such tangential symmetric spaces and 
obtain new tangential symmetric spaces which do not admit compact Clifford--Klein forms. 
As a result, in the class of irreducible classical semisimple symmetric spaces, we have only two types of symmetric spaces which are not proved not to admit compact Clifford--Klein forms.  


The existence problem of compact Clifford--Klein forms 
for homogeneous spaces of reductive type, which was initiated by T.~Kobayashi in 1980s, has been studied by various methods but is not completely solved yet. 
On the other hand, the one for tangential homogeneous spaces has been studied since 2000s and an analogous criterion was proved by T.~Kobayashi and T.~Yoshino. 
In concrete examples, further works are needed to verify Kobayashi--Yoshino's condition by direct calculations. 
In this paper, some easy-to-check necessary conditions(=obstructions) for the existence of compact quotients in the tangential setting are given, and they are applied to the case of symmetric spaces. 
The conditions are related to various fields of mathematics such as 
associated pair of symmetric space, 
Calabi--Markus phenomenon, 
trivializability of vector bundle (parallelizability, Pontrjagin class), 
Hurwitz--Radon number and 
Pfister's theorem (the existence problem of common zero points of polynomials of odd degree). 
\end{abstract}

\section{Introduction and main results}

In this paper, we give some obstructions to the existence of compact Clifford--Klein forms of tangential symmetric spaces and 
obtain new examples which do not admit compact Clifford--Klein forms. 

In Section~1, $G$ denotes a Lie group and $H$ denotes a closed subgroup of $G$. 
Geometry of Clifford--Klein forms has been enriched by the following: 
\begin{openproblem}[{\cite[Problem~1.7(2)]{kobayashi96}}]\label{openproblem1}
When does $G/H$ admit compact Clifford--Klein forms?
\end{openproblem}

This is still open 
even if we restrict the problem to irreducible semisimple symmetric spaces (see \S~\ref{setting} for the definition). 
A systematic study was initiated and Open Problem~\ref{openproblem1} was raised 
by T.~Kobayashi in 1980s. 
Most of the results are summarized in the papers \cite{kobayashi96,kob02,kobayoshi}. 

In this paper, we consider Problem~\ref{openproblem1} 
for tangential symmetric spaces $G_\theta/H_\theta$ (see Definition~\ref{def_of_tangential}) of irreducible classical semisimple symmetric spaces $G/H$. 
\begin{problem}[tangential case]\label{problem_tangential}
Classify irreducible classical semisimple symmetric spaces $G/H$ 
with regard to whether or not the tangential symmetric spaces 
$G_\theta/H_\theta$ admit compact Clifford--Klein forms. 
\end{problem}

For Problem~\ref{problem_tangential}, 
the following Facts~\ref{unpublished result}, \ref{intro_adams} and \ref{corollary from L} are known as  partial solutions: 
\begin{fact}[Kobayashi--Yoshino {\cite[Thm.~3]{kobayoshi2}}]\label{unpublished result}
Let $G/H$ be a classical irreducible symmetric space, and $G$ a complex reductive Lie group. 
$G_\theta/H_\theta$ has a compact Clifford--Klein form if and only if 
$G/H$ is locally isomorphic to one of the following list:
\begin{itemize}
\item Riemannian symmetric spaces $G/K$, 
\item Group manifolds $(G\times G)/(\diag_\tau G)$, where we put $\diag_\tau G:=\{(g,\tau(g)): g\in G\}\subset G\times G$ for each involution $\tau$ on $G$, 
\item $SO(8,\C)/SO(7,\C)$, 
\item $SO(8,\C)/SO(7,1)$. 
\end{itemize}
Here, the notion of irreducibility is in the sense of \cite[\S 16, Chapter III]{nomizu}.  
\end{fact}

\begin{fact}[Kobayashi--Yoshino {\cite[Prop.~5.5.1]{kobayoshi}}]\label{intro_adams}
The following conditions on the pair $(p,q)$ of positive integers are equivalent:
\begin{enumerate}
\item The tangential symmetric space of $SO_0(p,q+1)/SO_0(p,q)$ admits a compact Clifford--Klein form.
\item $q<\rho(p,\R)$. 
\end{enumerate}
Here, $\rho(p,\R)$ is the Hurwitz--Radon number (see Definition~\ref{Adams thm}). 
\end{fact}
The main parts of the above two facts are non-existence results for tangential symmetric spaces. 
On the other hand, existence results are also known. 
To state it, 
we introduce the notion of standard Clifford--Klein form: 
\begin{definition}[Kassel--Kobayashi {\cite[Def.~1.4]{koba-kassel}}]
Let $G$ be a linear reductive Lie group.  
A Clifford--Klein form $\Gamma\backslash G/H$ of $G/H$ is \emph{standard} 
if $\Gamma$ is contained in some reductive subgroup $L$ of $G$ acting properly on $G/H$. 
\end{definition}

\begin{remark}\label{stand_and_tang}
If there exists a standard compact Clifford--Klein form of $G/H$ of reductive type, 
then its tangential homogeneous space $G_\theta/H_\theta$ admits a compact Clifford--Klein form. 
That is, non-existence of compact Clifford--Klein forms of $G_\theta/H_\theta$ implies non-existence of standard compact Clifford--Klein forms of homogeneous space $G/H$ of reductive type. 
See Section~\ref{sec:preliminary} and {\cite[Sect. 5]{kobayoshi}} for more details. 
\end{remark}

\begin{remark}
Any tangential symmetric spaces associated with Riemannian symmetric spaces $G/K$ or group manifolds $(G\times G)/(\diag_\tau G)$ admit standard compact Clifford--Klein forms. 
Therefore, for Problem~\ref{problem_tangential}, we focus on the case where $G$ is simple and $H$ is not compact. 
\end{remark}

\begin{fact}[Kobayashi--Yoshino {\cite[Cor.~3.3.7]{kobayoshi}}]\label{corollary from L}
Let $(G, H)$ be a symmetric pair which is locally isomorphic to one in Table~1 and suppose that $G$ is connected. 
Then the tangential symmetric space $G_\theta/H_\theta$ of the symmetric space $G/H$ admits compact Clifford--Klein forms. 
\begin{center}
\stepcounter{table}
Table \thetable : Symmetric pairs $(G,H)$ which admit standard compact Clifford--Klein forms. 
  \scalebox{0.87}{
	\begin{tabular}{c|c|c||c|c|c}
 $G$ & $H$ & $L$ & $G$ & $H$ & $L$\\
\hline
  $SO_0(2,2n)$&$SO_0(1,2n)$ & $U(1,n)$&
  $SU(2,2n)$&$U(1,2n)$ & $Sp(1,n)$\\
\hline
 $SO_0(4,4n)$&$SO_0(3,4n)$ & $Sp(1,n)$&
 $SU(2,2n)$&$Sp(1,n)$ & $U(1,2n)$\\
\hline
 $SO_0(4,4)$&$SO_0(4,1)\times SO(3)$ & $Spin(4,3)$&
 $SO(8,\C)$&$SO(7,\C)$ & $Spin(1,7)$\\
\hline
 $SO(4,3)_0$&$SO_0(4,1)\times SO(2)$ & $G_{2(2)}$&
 $SO(8,\C)$&$SO(7,1)$ & $Spin(7,\C)$\\
\hline
 $SO_0(8,8)$&$SO_0(7,8)$ & $Spin(1,8)$&
 $SO^*(8)$&$SO^*(6)\times SO^*(2) $ & $Spin(1,6)$\\
\hline
 $SO_0(2,2n)$&$U(1,n)$ & $SO_0(1,2n)$&
 $SO^*(8)$&$U(3,1)$ & $Spin(1,6)$\\
	\end{tabular}
  }
\end{center}
Here, $L$ is a reductive subgroup of $G$ acting on $G/H$ properly and cocompactly. 
\end{fact}

\begin{remark}
For a symmetric pair $(G,H)$, 
both of implications between the conditions 
``Existence of compact Clifford--Klein forms for $G/H$'' and 
the condition ``Existence of compact Clifford--Klein forms for $G_\theta/H_\theta$'' 
have not been proved.
\end{remark}

In this paper, we give new examples which do not admit compact Clifford--Klein forms in the class of tangential symmetric spaces of irreducible semisimple symmetric spaces. 
To show the non-existence of compact Clifford--Klein forms of tangential symmetric spaces, 
it is enough to consider symmetric spaces up to associated pairs. 
This is one of the reasons why Problem~\ref{problem_tangential} is easier to deal with than the case where $G/H$ is of reductive type. 
We see it in Proposition~\ref{associated duality} in the next section.


We use the following five methods to give necessary conditions for the existence of compact Clifford--Klein forms of tangential symmetric spaces 
(Fact~\ref{Calabi-Markus}, Theorem~\ref{application_of_pfister}, Facts \ref{corollary_of_kob92} and \ref{non-triviality}): 
\begin{enumerate}
\item Calabi--Markus phenomenon \cite{koba:reductive,kobayoshi}, 
\item applications of Pfister's theorem, 
\item maximality of non-compactness \cite{koba:reductive,kob92,kobayoshi,kobayoshi2},
\item non-triviality of homogeneous spaces of reductive type as vector bundles \cite{kobayoshi2},
\item applications of Adams' theorem \cite{kobayoshi}. 
\end{enumerate}

The following tangential symmetric spaces $G_\theta/H_\theta$ are typical examples which are proved not to admit compact Clifford--Klein forms by each method: 
\begin{enumerate}
\item $SL(n,\C)_\theta/SL(n,\R)_\theta$ ($n\geq 2$), 
\item $SL(2n,\R)_\theta/Sp(n,\R)_\theta$ ($n\geq 2$), 
\item $SO_0(p_1+p_2,q_1+q_2)_\theta/(SO_0(p_1,q_1)\times SO_0(p_2,q_2))_\theta$ $(0<p_1\leq p_2,q_1,q_2)$, 
\item $SO_0(2p,2q)_\theta/U(p,q)_\theta$ ($2\leq p\leq q$), 
\item $SU(p,2)_\theta/U(p,1)_\theta$ ($p$ is odd). 
\end{enumerate}
Lists of non-existence results obtained by each method shall be given in Sections~\ref{section:CM}, \ref{section:Pfister}, \ref{section:maximality}, \ref{section:non-triviality} and \ref{section:adams}, respectively.  


\begin{theorem}\label{main_theorem1}
Let $(G, H)$ be a symmetric pair which is locally isomorphic to one in Tables~2 and 3 and suppose that $G$ is connected. 
Then the tangential symmetric space $G_\theta/H_\theta$ does not admit compact Clifford--Klein forms: \\
\begin{center}
\stepcounter{table}
Table \thetable : 
Symmetric pairs of which tangential symmetric spaces 
$G_\theta/H_\theta$ do not admit compact Clifford--Klein forms. \\
\begin{tabular}{c|c||c|c}
$G$ & $H$ & $G$ & $H$\\
\hline
$SL(p+q,\C)$&$S(GL(p,\C)\times GL(q,\C))$&$Sp(p+q,\C)$&$Sp(p,\C)\times Sp(q,\C)$   \\
$(p,q\geq 1)$&$SU(p,q)$                  &$(p,q\geq 1)$  &$Sp(p,q)$                   \\
\hline
$SL(n,\C)$  &$SL(n,\R)$                  &$Sp(n,\C)$  &$Sp(n,\R)$                  \\
$(n\geq 2)$  &$SO(n,\C)$                 & $(n\geq 1)$&$GL(n,\C)$                  \\
\hline
$SL(p+q,\R)$&$S(GL(p,\R)\times GL(q,\R))$&$Sp(p+q,\R)$&$Sp(p,\R)\times Sp(q,\R)$   \\
$(p,q\geq 1)$&$SO_0(p,q)$                 &$(p,q\geq 1)$&$U(p,q)$                    \\
\hline
$SU(p,q)$   &$SO_0(p,q)$                 &$Sp(p,q)$   &$U(p,q)$                    \\
$(p,q\geq 1)$ &                          &$(p,q\geq 1)$&                           \\
\hline
$SU(n,n)$   &$SL(n,\C)\times \R^\times$            &$SU(n,n)$   &$Sp(n,\R)$                  \\
$(n\geq 1)$ &                            &$(n\geq 2)$&$SO^*(2n)$                  \\
\hline
$SU^*(2n)$  &$SU^*(2n)\cap GL(2n,\R)$                 &$SO(2n,\C)$ &$GL(n,\C)$                  \\
$(n\geq 2)$ &$SO^*(2n)$                  &$(n\geq 2)$ &$SO^*(2n)$                  \\
\hline
$SU^*(2(p+q))$&$S(U^*(2p)\times U^*(2q))$&$Sp(n,\R)$  &$GL(n,\R)$                  \\
$(p,q\geq 1)$&$Sp(p,q)$                  &$(n\geq 1)$ &                            \\
\hline
$SO_0(n,n)$ &$GL(n,\R)$                  &$Sp(n,n)$   &$U^*(2n)$                   \\
$(n\geq 1)$ &$SO(n,\C)$                  &$(n\geq 1)$ &$Sp(n,\C)$                  \\
\hline
$SO^*(2n)$  &$SO(n,\C)$                  &$SO^*(4n)$  &$U^*(2n)$                   \\
$(n\geq 2)$ &                            &$(n\geq 1)$ &            
\end{tabular}
\end{center}

\newpage
\begin{center}
\stepcounter{table}
Table \thetable : 
Symmetric pairs $(G, H)$ of which tangential symmetric spaces $G_\theta/H_\theta$ do not admit compact Clifford--Klein forms. \\
\begin{tabular}{c|c|c}
$G$ & $H$ & condition \\
\hline
$SO^*(2(p+q))$ & $SO^*(2p)\times SO^*(2q)$  & $p\geq 2$ or $(q\neq 1$ and $q\neq 3)$\\
 & $U(p,q)$ $(1\leq p\leq q)$ & \\
\hline
$SO(p+q,\C)$ & $SO(p,\C)\times SO(q,\C)$  & $(p,q)\neq (1,1),(1,3),(1,7)$\\
 & $SO_0(p,q)$ $(1\leq p\leq q)$&\\
\hline
$SO_0(p,q)$ & $SO_0(p_1,q_1)\times SO_0(p_2,q_2)$  & $p_1\geq 1$ or $(q_1\geq 2$ and $q_2\geq 2)$ \\
&$(0\leq p_1\leq p_2, q_1, q_2\geq 1)$ & \\
\hline
$SU(p,q)$ & $S(U(p_1,q_1)\times U(p_2,q_2))$ & $p_1\geq 1$ or $q_1\geq 2$ or $q_2\geq 2$ or $p_2$ is odd. \\
&$(0\leq p_1\leq p_2, q_1, q_2\geq 1)$ &\\
\hline
$Sp(p,q)$ & $Sp(p_1,q_1)\times Sp(p_1,p_2)$  & $0\leq p_1\leq p_2, q_1, q_2\geq 1$\\
 &$(0\leq p_1\leq p_2, q_1, q_2\geq 1)$ & \\
\hline
$SL(2n,\C)$ & $Sp(n,\C)$ $(n\geq 2)$& $n\geq 2$ \\
            & $SU^*(2n)$ &\\
\hline
$SL(2n,\R)$ & $Sp(n,\R)$ $(n\geq 2)$ & $n\geq 2$ \\
 & $SU^*(2n)\cap GL(2n,\R)$ &\\
\hline
$SO_0(2p,2q)$ & $U(p,q)$ $(1\leq p\leq q)$ & $p\geq 2$
\end{tabular}
\end{center}

Here, for the three cases $SO_0(p,q)$, $SU(p,q)$ and $Sp(p,q)$, $p=p_1+p_2$ and $q=q_1+q_2$. 
\end{theorem}

Table~2 is a list of $(G,H)$ to which the method (i) is applicable (modulo the associated pair), 
and Table 3 is a list of $(G,H)$ whose non-existence of a compact quotient is proved by some other method. 

Theorem~\ref{main_theorem1} follows from Propositions~\ref{example_CM}, \ref{example_pfister}, \ref{example_maximality}, \ref{examples_non-triviality} and \ref{application_of_adams} which are stated in each section. 
\begin{remark}
As far as the author knows, Problem~\ref{problem_tangential} (including exceptional cases) is not solved for the following cases:
\begin{itemize}
\item ($G$, $H$) is locally isomorphic to ($Sp(2n,\R)$, $Sp(n,\C)$) for $n\geq 2$;
\item ($G$, $H$) is locally isomorphic to ($SU(2p,2q)$, $Sp(p,q)$) for $p,q\geq 2$; 
\item ($G$, $H$) is locally isomorphic to ($E_{6(-14)}, F_{4(-20)}$). 
\end{itemize}
In this paper, we do not discuss the cases where $G$ is exceptional. 
In those cases except for ($E_{6(-14)}, F_{4(-20)}$), by using the methods (i) and (iv), 
the papers \cite{tojo_excep} and \cite{tojo_dr} prove the non-existence of compact Clifford--Klein forms of tangential symmetric spaces. 
\end{remark}

\begin{remark}
$Sp(2,\R)_\theta/Sp(1,\C)_\theta$ does not admit compact Clifford--Klein forms. 
This comes from that symmetric pairs $(\frak{sp}(2,\R), \frak{sp}(1,\C))$ and $(\frak{so}(3,2), \frak{so}(3,1))$ are isomorphic to each other and that 
$SO_0(3,2)_\theta/SO_0(3,1)_\theta$ does not admit compact Clifford--Klein forms as a special case of Kobayashi--Yoshino's theorem (Fact~\ref{intro_adams}). 
\end{remark}

\section{Preliminary}\label{sec:preliminary}
In this paper, we consider tangential symmetric spaces $G_\theta/H_\theta$ associated with irreducible classical semisimple symmetric spaces $G/H$. 
We prepare the precise setting and notions in Subsection~\ref{setting}. 
In the next subsection, we review a criterion for the existence of compact Clifford--Klein forms for tangential homogeneous spaces given by Kobayashi--Yoshino \cite{kobayoshi}. 
In Subsection~\ref{subsec:associated_pair}, we see that it is enough to consider symmetric spaces up to associated pair for our problem.  
\subsection{Setting and notation}\label{setting}
Throughout this paper, we suppose that 
$G$ is a linear reductive and connected Lie group 
and $H$ is a reductive subgroup of $G$. 
If $H$ is an open subgroup of $G^\sigma:=\{g\in G : \sigma g=g\}$, 
where $\sigma$ is an involution of $G$, then such a pair $(G,H)$ is called a symmetric pair. 
Then the symmetric space $G/H$ is of reductive type ({\cite[Example~2.6.3]{kobayashi96}}). 
Moreover, if $G$ is semisimple, the pair $(G,H)$ and the space $G/H$ are said to be a semisimple symmetric pair and a semisimple symmetric space, respectively. 
We say that a semisimple symmetric space $G/H$ and a semisimple symmetric pair $(G,H)$ are irreducible if $G$ is simple or $G/H$ is locally isomorphic to a group manifold $(L\times L)/(\diag_\tau L)$, where $L$ is simple and $\diag_\tau L:=\{(\ell. \tau(\ell))\in L \times L\ |\ \ell \in L\}$ with an involution $\tau$ of $L$. 
We are mainly interested in the case of semisimple symmetric spaces in this paper. 

Now, we recall the definition of a tangential homogeneous space $G_\theta/H_\theta$ 
of a homogeneous space $G/H$ of reductive type.

\begin{definition}[Cartan motion group, Kobayashi--Yoshino {\cite[Sect.~5.1]{kobayoshi}}]\label{Cartan motion group} 
Let $\theta$ be a Cartan involution of $G$.
The Cartan motion group $G_\theta$ of $G$ is defined by
\[
G_\theta :=K\ltimes_{\Ad} \mfp.
\]
Here $K=G^\theta$ is a maximal compact Lie subgroup of $G$ 
and $\mfp=\mathfrak{g}^{-\theta}$. 
\end{definition}
Let $G/H$ be a homogeneous space of reductive type, then 
we can take a Cartan involution $\theta$ of $G$ such that 
$\theta|_H$ is also a Cartan involution of $H$. 
Then we get a closed subgroup $H_\theta:=K_H\ltimes \frak{p}_H$ of $G_\theta$ where $K_H=K\cap H$ and $\frak{p}_H=\frak{p}\cap \frak{h}$. 

\begin{definition}[Kobayashi--Yoshino {\cite[Def.~5.1.2]{kobayoshi}}]\label{def_of_tangential}
We call $(G/H)_\theta:=G_\theta/H_\theta$ the {\it tangential homogeneous space} of $G/H$. 
\end{definition}
\begin{remark}
If ($G$, $H$) is a symmetric pair, 
then so is ($G_\theta$, $H_\theta$). 
\end{remark}

\subsection{Tangential analogue of Kobayashi's criterion}\label{subsec:criterion}
By the following fact, 
the existence problem of compact Clifford--Klein forms 
for a tangential homogeneous space $G_\theta/H_\theta$ reduces to 
how large subspace of $\mfp$ satisfying condition Fact~\ref{yoshino's criterion}~(iii) we can take. 

\begin{fact}[Kobayashi--Yoshino {\cite[Thm.~5.3.2]{kobayoshi}}]\label{yoshino's criterion}
Let $G_\theta/H_\theta$ be a tangential homogeneous space of a homogeneous space $G/H$ of reductive type. 
Then, the following three conditions are equivalent:
\begin{enumerate}
\item The homogeneous space $G_\theta/H_\theta$ admits compact Clifford--Klein forms. 
\item There exists a subspace $V$ in $\frak{p}$ satisfying the following two conditions (a) and (b). 
\begin{itemize}
\item[(a)] $\frak{a}(V)\cap \frak{a}(H)=\{0\}$, 
\item[(b)] $\dim V+d(H)=d(G)$. 
\end{itemize}
\item There exists a subspace $V$ in $\frak{p}$ satisfying the following two conditions (a)' and (b)
\begin{itemize}
\item[(a)'] $V \cap \Ad(K)\frak{p}_H=\{0\}$,
\item[(b)] $\dim V+d(H)=d(G)$. 
\end{itemize}
\end{enumerate}
Here, $\frak{a}$ is a fixed maximal abelian subspace of $\frak{p}$ and 
$d(G)=\dim \frak{p}$, $d(H)=\dim \frak{p}_H$ are non-compact dimensions of $G$, $H$, respectively (\cite{koba:reductive}). 
For a subset $L$ in the Cartan motion group $G=K\ltimes \frak{p}$, 
we put $\frak{a}(L):=KLK\cap \frak{a}$. 
\end{fact}

\begin{remark}
If $G$ is connected, 
the existence problem of compact Clifford--Klein forms 
for the tangential symmetric space of a homogeneous space of reductive type $G/H$ depends only on the local isomorphism class 
of $(G, H)$ because of Fact~\ref{yoshino's criterion}. 
\end{remark}

Remark~\ref{stand_and_tang} follows from 
the following Facts~\ref{criterion_of_properness} and 
\ref{criterion_of_cocompactness} by taking $\frak{p}_L$ as $V$ in Fact~\ref{yoshino's criterion}. 
\begin{fact}[Kobayashi {\cite[Thm.~4.1]{koba:reductive}}]\label{criterion_of_properness}
Let $H$ and $L$ be reductive subgroups of a linear reductive Lie group $G$. 
Then the following conditions on $H$ and $L$ are equivalent:
\begin{enumerate}
\item The $L$-action on $G/H$ is proper, 
\item $\Ad(K)\frak{p}_H\cap \frak{p}_L=\{0\}$. 
\end{enumerate}
\end{fact}
\begin{fact}[Kobayashi {\cite[Thm.~4.7]{koba:reductive}}]\label{criterion_of_cocompactness}
Let $H$ and $L$ be reductive subgroups of a linear reductive Lie group $G$. 
Under the conditions in Fact~\ref{criterion_of_properness}, 
the following conditions are equivalent: 
\begin{enumerate}
\item The double coset space $L\backslash G/H$ is compact, 
\item $d(G)=d(H)+d(L)$. 
\end{enumerate}
\end{fact}

\subsection{Associated pair}\label{subsec:associated_pair}
In this subsection,  
we review that the existence problem of compact Clifford--Klein forms 
for the tangential symmetric space of a symmetric space $G/H$ is equivalent to 
that for the tangential symmetric space of $G/H^a$, which is due to Kobayashi--Yoshino \cite{kobayoshi2},
where $(G,H^a)$ is the associated pair of $(G, H)$ defined as follows: 

Let $\sigma$ be an involution of $G$ 
which defines the symmetric pair $(G, H)$. 
We take a Cartan involution $\theta$ of $G$ satisfying $\theta \circ \sigma = \sigma \circ \theta$. 
Then $\theta|_H$ is also a Cartan involution of $H$ 
and $\theta \circ \sigma$ is also an involution of $G$. 

\begin{definition}[\cite{berger}, see also \cite{oshima_sekiguchi}]\label{def of associated pair}
We call the symmetric pair $(G,H^a)$ defined 
by $\theta \circ \sigma$ the {\it associated pair} of $(G,H)$. 
\end{definition}

\begin{remark}
For a given symmetric pair $(G,H)$, $H^a$ is not unique. 
However the space $G/H^a$ is unique as a symmetric space up to a local isomorphism. 
\end{remark}

By the definition of $\theta$ and $\sigma$, one can easily see that 
the associated pair of $(G, H^a)$ is $(G, H)$. 

\begin{proposition}[Kobayashi--Yoshino {\cite[Thm.~20]{kobayoshi2}}]\label{associated duality}
Let $(G,H)$ be a semisimple symmetric pair 
and $(G,H^a)$ the associated pair of $(G,H)$. 
Then $G_\theta/H_\theta$ admits compact Clifford--Klein forms 
if and only if $G_\theta/H^a_\theta$ admits compact Clifford--Klein forms. 
\end{proposition}

\begin{proof}
It is enough to show ``only if'' part. 
Suppose $s(G,H)=d(G)-d(H)(=d(H^a))$. 
Let $B$ be the restriction on $\mfp$ of the Killing form on $\frak{g}$. 
Take a subspace $V$ in $\frak{p}$ such that 
$\dim V=d(G)-d(H)=d(H^a)$ and $\Ad(K)\frak{p}_H\cap V=\{0\}$. 
By taking the orthogonal complement of $V$, 
we obtain the subspace $V^\perp$ in $\mfp$ satisfying the two conditions, $\dim V^\perp=d(G)-d(H^a)$ and $\Ad(K)\mfp_{H^a}\cap V^\perp=\{0\}$, 
which follow from the facts that the representation $\Ad$ is unitary and that the orthogonal complement of $\mfp_H$ is $\mfp_{H^a}$ with regard to $B$. 
\end{proof}

\section{Calabi--Markus phenomenon}\label{section:CM}
We recall from Kobayashi \cite{koba:reductive} and Kobayashi--Yoshino \cite{kobayoshi} that Calabi--Markus phenomenon is one of the necessary conditions for the existence of compact quotients in the reductive setting, established by Kobayashi \cite{koba:reductive}, and 
that there is a tangential analogue of the Calabi--Markus phenomenon given by Kobayashi--Yoshino \cite{kobayoshi}. 
\begin{fact}[Kobayashi--Yoshino {\cite[Sect. 5.2 (8)]{kobayoshi}}]\label{Calabi-Markus}
If a homogeneous space $G/H$ of reductive type satisfies that 
$\rank_\R G=\rank_\R H$ and $G/H$ is non-compact, 
then the tangential homogeneous space $G_\theta/H_\theta$ 
does not admit compact Clifford--Klein forms. 
\end{fact}


\begin{remark}
For a homogeneous space $G/H$ of reductive type, Kobayashi \cite[Cor. 4.4]{koba:reductive} established a necessary and sufficient condition for the Calabi--Markus phenomenon by the real rank condition $\rank_\R G=\rank_\R H$. 
Later, Kobayashi--Yoshino \cite{kobayoshi} showed that an analogous result holds for the tangential homogeneous space $G_\theta/H_\theta$, namely, 
only a finite subgroup of $G_\theta$ can act properly discontinuously on $G_\theta/H_\theta$, if and only if $\rank_\R G=\rank_\R H$ by an easier and analogous proof, see \cite[Sect. 5.2]{kobayoshi}. 
\end{remark}

For a irreducible classical semisimple symmetric pair $(G,H)$, 
we consider the following two conditions A and B. 
\begin{align*}
\text{A}&:\rank_\R G=\rank_\R H,\\
\text{B}&:\text{the associated pair satisfies the condition A.}
\end{align*}
\begin{proposition}\label{example_CM}
Let $(G,H)$ be a symmetric pair which is locally isomorphic to one in Table~4 and 
suppose that $G$ is connected. 
Then $G_\theta/H_\theta$ does not admit compact Clifford--Klein forms. 
\end{proposition}
\begin{center}
\stepcounter{table}
Table \thetable : Symmetric pairs $(G,H)$ satisfying A or B. 
\begin{tabular}{c|c|c|c|c}
$G$&$H$&$\rank_\R G$&$\rank_\R H$ &A or B\\
\hline
$SL(p+q,\C)$ 
&$S(GL(p,\C)\times GL(q,\C))$&$p+q-1$&$p+q-1$&A\\
$p,q\geq 1$&$SU(p,q)$&&$\min(p,q)$&B\\
\hline
$SL(n,\C)$
&$SL(n,\R)$&$n-1$&$n-1$&A\\
$n\geq 2$&$SO(n,\C)$&&$\lfloor\frac{n}{2}\rfloor$&B\\
\hline
$SL(p+q,\R)$ 
&$S(GL(p,\R)\times GL(q,\R))$&$p+q-1$&$p+q-1$&A\\
$p,q\geq 1$&$SO_0(p,q)$&&$\min(p,q)$&B\\
\hline
$SO(2n,\C)$
&$GL(n,\C)$&$n$&$n$&A\\
$n\geq 2$&$SO^*(2n)$&&$\lfloor\frac{n}{2}\rfloor$&B\\
\hline
$SO_0(n,n)$
&$GL(n,\R)$&$n$&$n$&A\\
$n\geq 1$&$SO(n,\C)$&&$\lfloor\frac{n}{2}\rfloor$&B\\
\hline
$SU(n,n)$ $n\geq 1$
&$SL(n,\C)\times \R^\times $&$n$&$n$&A\\
\hline
$SU(n,n)$
&$Sp(n,\R)$&$n$&$n$&A\\
$n\geq 1$&$SO^*(2n)$&&$\lfloor\frac{n}{2}\rfloor$&B\\
\hline
$SU(p,q)$ $p,q\geq 1$
&$SO_0(p,q)$&$\min(p,q)$&$\min(p,q)$&A\\
\hline
$Sp(n,n)$
&$U^*(2n)$&$n$&$n$&A\\
$n\geq 1$&$Sp(n,\C)$&&$n$&A\\
\hline
$Sp(p,q)$
 $p,q\geq 1$&$U(p,q)$&$\min(p,q)$&$\min(p,q)$&A\\
\hline
$Sp(p+q,\C)$ 
&$Sp(p,\C)\times Sp(q,\C)$&$n$&$n$&A\\
$p,q\geq 1$&$Sp(p,q)$&&$\min(p,q)$&B\\
\hline
$Sp(n,\C)$
&$Sp(n,\R)$&$n$&$n$&A\\
$n\geq 1$&$GL(n,\C)$&&$n$&A\\
\hline
$Sp(n,\R)$ $n\geq 1$
&$GL(n,\R)$&$n$&$n$&A\\
\hline
$Sp(p+q,\R)$ $p,q\geq 1$
&$Sp(p,\R)\times Sp(q,\R)$&$n$&$n$&A\\
&$U(p,q)$&&$\min(p,q)$&B\\
\hline
$SU^*(2n)$ 
&$SU^*(2n)\cap GL(2n,\R)$&$n-1$&$n-1$&A\\
$n\geq 2$&$SO^*(2n)$&&$\lfloor\frac{n}{2}\rfloor$&B\\
\hline
$SU^*(2(p+q))$ 
&$S(U^*(2p)\times U^*(2q))$&$n-1$&$n-1$&A\\
$p,q\geq 1$&$Sp(p,q)$&&$\min(p,q)$&B\\
\hline
$SO^*(2n)$
&$SO(n,\C)$&$\lfloor\frac{n}{2}\rfloor$&$\lfloor\frac{n}{2}\rfloor$&A\\
\hline
$SO^*(4n)$ $n\geq 1$
&$U^*(2n)$&$n$&$n$&A\\
\hline 
$SO^*(2(p+q))$ $p,q\geq 1$& $SO^*(2p)\times SO^*(2q)$ & $\lfloor \frac{p+q}{2}\rfloor$ & $\lfloor\frac{p}{2}\rfloor+\lfloor\frac{q}{2}\rfloor$ & A \\
$p$ or $q$ is even & $U(p,q)$ & &$\min(p,q)$ & B \\
\hline
$SO(p+q, \C)$ $p,q\geq 1$ & $SO(p,\C)\times SO(q,\C)$ & $\lfloor \frac{p+q}{2}\rfloor$ & $\lfloor\frac{p}{2}\rfloor+\lfloor\frac{q}{2}\rfloor$ & A \\
$p$ or $q$ is even & $SO_0(p.q)$ & & $\min(p,q)$ & B
\end{tabular}
\label{tab:excluded}
\end{center}


The remaining cases are listed in Table~5. 
We also enrich the table with methods for the proof of Theorem~\ref{main_theorem1}. 
\begin{center}
\stepcounter{table}
Table \thetable: irreducible semisimple symmetric pairs we consider in the following sections. 
\begin{tabular}{c|c|c}
$G$&$H$&methods\\
\hline
$SO^*(2(2p+2q+2))$ 
&$SO^*(2(2p+1))\times SO^*(2(2q+1))$&(iii), (v)\\
$p,q\geq 0$, $(p,q)\neq (0,0)$&$U(2p+1,2q+1)$& \\
\hline
$SO(2p+2q+2,\C)$ 
&$SO(2p+1,\C)\times SO(2q+1,\C))$&(v)  \\
$p,q\geq 0$, $(p,q)\neq (0,0)$&$SO_0(2p+1,2q+1)$&\\
\hline
$SO_0(p_1+p_2,q_1+q_2)$ 
&$SO_0(p_1,q_1)\times SO_0(p_2,q_2)$& (iii),  (iv), (v) \\
$0\leq p_1\leq p_2, q_1, q_2\geq 1$ && \\
\cline{1-3}
$SU(p_1+p_2,q_1+q_2)$ 
&$S(U(p_1,q_1)\times U(p_2,q_2))$&(iii), (iv), (v)\\
$0\leq p_1\leq p_2, q_1, q_2\geq 1$&& \\
\cline{1-3}
$Sp(p_1+p_2,q_1+q_2)$ 
&$Sp(p_1,q_1)\times Sp(p_2,q_2)$&(iii), (iv) \\
$0\leq p_1\leq p_2, q_1, q_2\geq 1$&&\\
\hline
$SL(2n,\C)$
&$Sp(n,\C)$&(ii) \\
$n\geq 2$&$SU^*(2n)$&\\
\cline{1-3}
$SL(2n,\R)$
&$Sp(n,\R)$&(ii) \\
$n\geq 2$&$SU^*(2n)\cap GL(2n,\R)$&\\
\hline
$SO_0(2p,2q)$ $1\leq p\leq q$
&$U(p,q)$&(iv) \\
\end{tabular}
\end{center}

\section{Applications of Pfister's theorem}\label{section:Pfister}
In this section, 
we give a necessary condition for the existence of compact Clifford--Klein forms for tangential symmetric spaces (Theorem~\ref{application_of_pfister}) 
and apply it to two types of symmetric pairs (Proposition~\ref{example_pfister}). 
We use Pfister's theorem (see Fact~\ref{Fulton's Lemma}) to prove Theorem~\ref{application_of_pfister}. 

\begin{theorem}\label{application_of_pfister}
Let $G/H$ be a semisimple symmetric space and fix an embedding of Lie algebras $\frak{g}\subset \frak{sl}(2n,\K)$, where $\K=\R$ or $\C$. 
If the following two conditions are satisfied, 
then $G_\theta/H_\theta$ does not admit compact Clifford--Klein forms. 
\begin{enumerate}
\item $d(G)-d(H)\geq n$, 
\item For $X\in \frak{a}\subset M(2n,\K)$, if the characteristic polynomial of $X$ over $\K$ is even, then $X$ is contained in $W_G\cdot \frak{a}_H$. 
\end{enumerate}
Here, $W_G$ is the Weyl group of $G$ and $\frak{a}_H:=\frak{a}\cap \frak{h}$. 
\end{theorem}

\begin{fact}[{\cite{pfister}}. See also {\cite[Example~13.1(c)]{fulton}}]\label{Fulton's Lemma}
Let $V$ be a real vector space with dimension $n+1$. 
Suppose $f_i : V \rightarrow \R$ ($i=1,\cdots, n$) are homogeneous polynomial functions on $V$ of odd degree. 
Then $\{f_i\}^n_{i=1}$ has common zero points in $V\setminus\{0\}$. 
\end{fact}

\begin{proof}[Proof of Theorem~\ref{application_of_pfister}]
From Fact~\ref{yoshino's criterion} and $W_G\cdot \frak{a}_H\subset \Ad(K)\frak{p}_H$, it is enough to prove that for any $\R$-subspace $V$ with dimension $n$ of $\frak{p}$, 
$W_G\cdot \frak{a}_H\cap V\neq \{0\}$ holds. 
By the assumption, it is enough to show that there exists a non-zero element $X\in V$ such that $f_X(x)$ is even, where $f_X$ denotes the characteristic polynomial of $X\in V\subset M(2n,\K)$. 
Let $V$ be a subspace of $\frak{p}$ such that $\dim_\R V=n$. 
We define maps $\tau_i:V\to \R$ ($i=0,1,\cdots,2n$) by 
\[ f_X(x)=\det(xI-X)=\sum_{i=0}^{2n}\tau_i(X)x^i \text{ for }X\in V. \]
Then $\tau_{2n}=1$, $\tau_{2n-1}(X)=\trace (X)=0$ for all $X\in V$ by definition. 
Since $\tau_{2i-1}$ ($i=1,2,\cdots, n-1$) are homogeneous polynomials on $V$ of odd degree, by using the Fact~\ref{Fulton's Lemma}, 
we can take a non-zero element $X\in V$ such that $f_X(x)$ is even. 
\end{proof}

\begin{proposition}\label{example_pfister}
Let $(G, H)$ and $(G, H^a)$ be symmetric pairs which are locally isomorphic to one of the following list and suppose that $G$ is connected. 
Then neither $G_\theta/H_\theta$ nor $G_\theta/H^a_\theta$ 
admit compact Clifford--Klein forms. 
\begin{itemize}
\item $(G,H,H^a)=(SL(2n,\R), Sp(n,\R), SU^*(2n)\cap GL(2n,\R))$ $(n\geq 2)$, 
\item $(G,H,H^a)=(SL(2n,\C), Sp(n,\C), SU^*(2n))$ $(n\geq 2)$.
\end{itemize}
\end{proposition}

\begin{proof}
Let $(G,H)$ be a symmetric pair $(SL(2n,\K),Sp(n,\K))$ where $K=\R$ or $\C$. 
Our claim follows from Theorem~\ref{application_of_pfister}, Lemma~\ref{p-even} below and the inequality $n\leq d(G)-d(H)=\begin{cases}n^2-1\ (\K=\R),\\ 2n^2-n-1\ (\K=\C)\end{cases}$ for $n\geq 2$. 
\end{proof}

We consider a symmetric pair $(G, H)=(SL(2n,\K), Sp(n,\K))$, where $\K=\R$ or $\C$. 
We realize a symmetric pair $(SL(2n,\K), Sp(n,\K))$ as follows. 
\begin{align*}
SL(2n,\K)&=\{g\in GL(2n,\K): \det g=1\}, \\
Sp(n,\K)&=\{g\in SL(2n,\K): {}^tgJ_ng=J_n\},
\end{align*}
where $J_n=\begin{pmatrix}0&-I_n\\I_n&0\end{pmatrix}$. 
Then, by taking a Cartan involution $\theta: g\mapsto \trans \overline{g}^{-1}$, we have
\begin{align*}
K&=\begin{cases}
SO(2n)\ (\K=\R),\\
SU(2n)\ (\K=\C),
\end{cases}\\
\mfp&=\Herm_0(2n,\K)=\{X\in M(2n,\K): {}^t\overline{X}=X, \trace X=0\},\\ 
\mfp_H&=\left\{\begin{pmatrix}A&B\\\overline{B}&-\overline{A}\end{pmatrix}:A\in \Herm(n,\K), B\in \Sym(n,\K)\right\}.
\end{align*}
We take a maximal abelian subspace $\frak{a}$ of $\frak{p}$ as follows. 
\[\frak{a}:=\left\{\diag(b_1,\cdots, b_{2n}): b_i\in \R\ (i=1,\cdots, 2n), \sum_{i=1}^{2n}b_i=0\right\}.  \]
Then we have 
\[
\mathfrak{a}_H=\{\operatorname{diag}(a_1,\cdots, a_n, -a_1,\cdots, -a_n):a_i\in \R\ (i=1,\cdots, n)\}. 
\]
\begin{lemma}\label{p-even}
For $X\in \frak{a}$, 
the following conditions are equivalent:
\begin{enumerate}
\item $X\in W_G\cdot \frak{a}_H$,
\item the characteristic polynomial $f_X(x)=\det(xI-X)$ is even. 
\end{enumerate}
\end{lemma}

\begin{proof}
It is clear from the fact that a polynomial $(x-b_1)\cdots (x-b_{2n})$ $(b_1,\cdots, b_{2n}\in \R)$ is even if and only if 
$(b_{\sigma(1)},\cdots,b_{\sigma(2n)})=(a_1,\cdots, a_n, -a_1,\cdots, -a_{n})$ for some $a_1,\cdots, a_n\in \R$ and $\sigma\in \frak{S}_{2n}\simeq W_G$. 
\end{proof}

\section{Maximality of non-compactness}\label{section:maximality}
In this section, we apply the idea of Kobayashi's comparison theorem \cite{kob92} to the tangential homogeneous space $G_\theta/H_\theta$. 
As a result, we show the non-existence of compact Clifford--Klein forms for six types of tangential symmetric spaces (Proposition~\ref{example_maximality}).

\begin{fact}[Kobayashi's comparison theorem {\cite[Thm.~1.5]{kob92}}]\label{maximality_kob}
Let $G/H$ be a homogeneous space of reductive type. 
If there exist a closed subgroup $H'$ reductive in $G$ 
satisfying the following two conditions, 
then $G/H$ does not admit compact Clifford--Klein forms. 
\begin{enumerate}
\item $\frak{a}_{H'}\subset W_G\cdot\frak{a}_H$, 
\item $d(H')>d(H)$. 
\end{enumerate}
Here, $W_G:=N_G(\frak{a})/Z_G(\frak{a})$ is the Weyl group of $G$, and $\frak{a}_H:=\frak{a}\cap \frak{h}$. 
\end{fact}

\begin{fact}[a tangential analogue of Kobayashi's comparison theorem \cite{kobayoshi2}]\label{corollary_of_kob92}
If a homogeneous space $G/H$ of reductive type satisfies the assumption of Fact~\ref{maximality_kob}, 
then the tangential homogeneous space $G_\theta/H_\theta$ 
does not admit compact Clifford--Klein forms. 
\end{fact}
The proof is the same as in \cite{kob92}. 

\begin{proposition}\label{example_maximality}
Let $(G, H)$ be a symmetric pair which is locally isomorphic to one of the following list and suppose that $G$ is connected. 
Then the tangential symmetric space $G_\theta/H_\theta$ does not 
admit compact Clifford--Klein forms. 
\begin{itemize}
\item $(G,H)=(SO^*(2(p+q)), SO^*(2p)\times SO^*(2q))$ $(2\leq p,q)$,
\item $(G,H)=(SO^*(2(p+q)), U(p,q))$ $(2\leq p,q)$, 
\item $(G,H)=(SO_0(p,q),SO_0(p_1,q_1)\times SO_0(p_2,q_2))$ $(0<p_1,p_2,q_1,q_2)$,
\item $(G,H)=(SU(p,q),S(U(p_1,q_1)\times U(p_2,q_2)))$ $(0<p_1,p_2,q_1,q_2)$, 
\item $(G,H)=(Sp(p,q),Sp(p_1,q_1)\times Sp(p_2,q_2))$ $(0<p_1,p_2,q_1,q_2)$, 
\item $(G,H,H^a)=(SO(2(p+q)-2,\C),SO(2p-1,\C)\times SO(2q-1,\C), SO_0(2p-1,2q-1))$ $(2\leq p\leq q)$.
\end{itemize}
\end{proposition}
\begin{proof}
This comes from Fact~\ref{corollary_of_kob92} and the fact that 
($SO^*(2(p+q))$, $U(p,q)$) is the associated pair of ($SO^*(2(p+q))$, $SO^*(2p)\times SO^*(2q)$). 
See {\cite[Example~1.7, 1.9]{kob92}}, {\cite[Remark~3.5.8, Corollary~3.5.9]{kobayoshi}}. 
\end{proof}

Fact~\ref{corollary_of_kob92} can be generalized as follows: 
\begin{fact}[Corollary to Kobayashi--Yoshino {\cite[Thm.~5.3.2]{kobayoshi}}]
If there exists a linear subspace $W$ of $\frak{p}$ such that 
$W\subset \Ad(K)\frak{p}_H$ and $\dim W>\dim\frak{p}_H$, 
then $G_\theta/H_\theta$ does not admit compact Clifford--Klein forms. 
\end{fact}

\section{Non-triviality of homogeneous spaces of reductive type as vector bundles}\label{section:non-triviality}

\subsection{General method}
In this subsection, 
we review the following necessary condition for the existence of compact Clifford--Klein forms of tangential homogeneous spaces (Fact~\ref{non-triviality}), which is due to Kobayashi--Yoshino \cite{kobayoshi2}, and 
give some examples which are proved not to admit compact Clifford Klein forms by using Fact~\ref{non-triviality} (Proposition~\ref{examples_non-triviality}). 
\begin{fact}[Kobayashi--Yoshino{\cite[Thm. 18]{kobayoshi2}}]\label{non-triviality}
Let $G/H$ be a homogeneous space of reductive type. 
If the associated vector bundle $K\times_{K_H}\frak{p}/\frak{p}_H\to K/K_H$ is not a trivial bundle, 
then the tangential homogeneous space $G_\theta/H_\theta$ does not admit compact Clifford--Klein forms. 
\end{fact}

\begin{proof}
This follows from the Fact~\ref{yoshino's criterion} and Lemma~\ref{inj bundle map} stated below by taking $(\sigma, V)=(\Ad, \frak{p})$ and $W_1=\frak{p}_H$. 
\end{proof}

\begin{fact}[Kobayashi {\cite[Lem.~2.7]{koba:reductive}}]
Let $G/H$ be a homogeneous space of reductive type. 
Then, there is a diffeomorphism $G/H\simeq K\times_{K_H}\mfp/\mfp_H$ as a manifold. 
\end{fact}

The following Lemma~\ref{inj bundle map} is used in the proof of Fact~\ref{non-triviality}. 
\begin{lemma}\label{inj bundle map}
Let $K$ be a Lie group, $K_H$ a closed subgroup of $K$ 
and $(\sigma,V)$ a finite dimensional representation of $K$. 
Let $W_1$ be a $\sigma(K_H)$-invariant subspace of $V$ 
and $W_2$ a subspace of $V$ satisfying $\sigma(K)W_1\cap W_2=\{0\}$. 
Then there exists an injective bundle map 
$K/K_H \times W_2 \hookrightarrow K\times _{K_H} V/W_1$ over $K/K_H$.
In particular, in the case where $\dim W_1+\dim W_2=\dim V$, 
the vector bundle $K\times_{K_H} V/W_1$ is trivial. 
\end{lemma}

\begin{remark}
In the above Lemma~\ref{inj bundle map}, 
the coefficient field of vector spaces $V, W$ can be considered as both $\R$ and $\C$. 
Moreover, we may just assume that $K$ is a topological group instead of a Lie group. 
\end{remark}

\begin{proof}[Proof of Lemma~\ref{inj bundle map}]
We define a bundle map $\tau$ over $K$ by 
\[ \tau: K\times W_2\to K\times V/W_1,\ (k,w_2)\mapsto (k,\sigma(k^{-1})w_2+W_1). \]
Then $\tau$ is injective and $K_H$-equivariant. 
Here, the right $K_H$-actions are given as follows:
\begin{align*}
(K\times W_2) \times K_H &\to K\times W_2,\quad ((k,w_2),h)\mapsto (kh, w_2),\\
(K\times V/W_1)\times K_H&\to K\times V/W_1, ((k,v+W_1),h)\mapsto (kh,\sigma(h^{-1})v+W_1).
\end{align*}
Therefore, we get the induced bundle map $\tilde{\tau}: K/K_H\times W_2\to K\times_{K_H}V/W_1$ over $K/K_H$, which is the desired injective bundle map. 
\end{proof}

In this subsection, we show that the tangential symmetric spaces associated with the following symmetric pairs do not admit compact Clifford--Klein forms. 
\begin{proposition}\label{examples_non-triviality}
Let $p,q_1$ and $q_2$ be positive integers. 
Suppose that 
$(G, H)$ a symmetric pair which is locally isomorphic to one of the following list and that $G$ is connected. 
Then $G_\theta/H_\theta$ does not 
admit compact Clifford--Klein forms. 
\begin{itemize}
\item $(G,H)=(SO_0(p,q_1+q_2), SO(q_1)\times SO_0(p,q_2))$ $(q_1\geq 2$ and $q_2\geq 2)$,
\item $(G,H)=(SU(p,q_1+q_2), S(U(q_1)\times U(p,q_2)))$ $(q_1\geq 2$ or $q_2\geq 2)$, 
\item $(G,H)=(Sp(p,q_1+q_2), Sp(q_1)\times Sp(p,q_2))$ $(p\geq 1$, $q_1, q_2\geq 1)$, 
\item $(G,H)=(SO_0(2p,2q), U(p,q))$ $(2\leq p, q)$. 
\end{itemize}
\end{proposition}
\begin{proof}
This follows from Facts~\ref{non-triviality} and \ref{pont_is_obst}, and Propositions~\ref{first Pont} and \ref{Pont of SO(2p,2q)/SU(p,q)}. 
\end{proof}

To show the non-triviality of real vector bundle, 
we prove the first Pontrjagin class does not vanish. 
Indeed, we have  
\begin{fact}[See \cite{dupont} for example]\label{pont_is_obst}
Let $E\to M$ be a real vector bundle. 
If the $i$-th Pontrjagin class $p_i(E)\in H^{4i}_{DR}(M,\R)$ does not vanish for some $i\geq 1$, 
then the bundle $E\to M$ is not trivial. 
\end{fact}

In the next subsection, we show the following Propositions~\ref{first Pont} and \ref{Pont of SO(2p,2q)/SU(p,q)}:
\begin{proposition}\label{first Pont}
Let $p, q_1, q_2$ be positive integers. 
\begin{enumerate}
\item[(a)] Let $(G$, $H)$$=$$(SO_0(p,q_1+q_2)$, $SO(q_1)\times SO_0(p,q_2))$. 
Then the first Pontrjagin class of the vector bundle $K\times_{K_H}\frak{p}/\frak{p}_H\to K/K_H$ does not vanish if and only if $q_1\geq 2$ and $q_2\geq 2$. 
\item[(b)] Let $(G$, $H)$$=$$(SU(p,q_1+q_2)$, $S(U(q_1)\times U(p,q_2)))$. 
Then the first Pontrjagin class of the vector bundle $K\times_{K_H}\frak{p}/\frak{p}_H\to K/K_H$ does not vanish if and only if $q_1\geq 2$ or $q_2 \geq 2$. 
\item[(c)] Let $(G,H)=(Sp(p,q_1+q_2),Sp(q_1)\times Sp(p,q_2))$. 
The first Pontrjagin class of the vector bundle $K\times_{K_H}\frak{p}/\frak{p}_H\to K/K_H$ does not vanish. 
\end{enumerate}
\end{proposition}

\begin{proposition}\label{Pont of SO(2p,2q)/SU(p,q)}
Let $(G$, $H)=(SO_0(2p,2q)$, $U(p,q))$ $(1\leq p\leq q)$. 
Then the first Pontrjagin class of the vector bundle $K\times_{K_H}\mfp/\mfp_H\to K/K_H$ does not vanish if and only if $p\geq 2$. 
\end{proposition}

\begin{remark}
Propositions~\ref{first Pont} and \ref{Pont of SO(2p,2q)/SU(p,q)} are probably well-known to the topologists. 
However, the author could not find an appropriate reference, so we give an outline of the proofs. 
\end{remark}

\begin{remark}
In the reductive case, Kobayashi--Ono \cite{kobayashi-ono} discovered an obstruction to the existence of compact quotients using de Rham cohomology and characteristic classes. 
This idea was repeatedly used in Morita \cite{morita} and Tholozan \cite{tholozan}. 
They also did some computations similar (but not directly linked) to the proofs of Propositions \ref{first Pont} and \ref{Pont of SO(2p,2q)/SU(p,q)}. 
\end{remark}

By using the following fact, 
we can easily calculate the Pontrjagin class of associated bundles. 
\begin{fact}[See \cite{dupont} for example.]\label{moritamethod}
Let $G$ be a connected compact Lie group, 
$\varpi:P\rightarrow M $ a principal $G$-bundle, 
$\rho : G\rightarrow SO(V)$ a representation of $G$ and 
$E:=P\times _G V$ the associated bundle. 
Then for any $f\in S^k(\frak{so}(V)^*)^{SO(V)}$, 
the following equality holds. 
\[ [f(R)]= \omega \circ d\rho^* (f)\in H^{2k}_{DR}(M,\R), \]
where $R$ is a curvature on $E$ 
and $\omega : S(\frak{g}^*)^G\rightarrow H^*_{DR}(M,\R)$ is the Chern--Weil map. 
\end{fact}
To determine whether or not the Pontrjagin class vanishes, we use the following: 
\begin{fact}[{\cite{cartan}}, See also {\cite{supersymmetry}}]\label{fact of ker omega}
Let $K$ be a connected compact Lie group and $K_H$ its closed connected subgroup.
Let $\omega:S(\mathfrak{t}_{H}^*)^{W_H} \rightarrow H_{DR}^*(K/K_H;\R)$ be the Chern--Weil map. Then $\operatorname{ker}\omega$ can be written as follows.
\[
\operatorname{ker}\omega=
(\text{ideal generated by } \bigoplus _{k=1}^\infty \operatorname{Im}(\operatorname{rest} :S^k(\mathfrak{t}^*)^W\rightarrow S^k(\mathfrak{t}_H^*)^{W_H})\ \text{in}\ S(\mathfrak{t}_H^*)^{W_H}), 
\]
where $\operatorname{rest} $ is the restriction map, and $\frak{t}$, $\frak{t}_H$ are Lie algebras of the maximal tori of $K$, $K_H$, respectively. $W$, $W_H$ are the Weyl groups of $K$, $K_H$, respectively. 
\end{fact}

\subsection{Calculation of first Pontrjagin class for Grassmannian manifolds}\label{calc of Pont}
In this subsection, we give outlines of proofs of Propositions~\ref{first Pont} and \ref{Pont of SO(2p,2q)/SU(p,q)}. 
Here we use Fact~\ref{moritamethod} and \ref{fact of ker omega} to check whether the first Pontrjagin class vanishes or not. 
\begin{lemma}\label{lemma:vector_bundle}
Let $\K=\R$, $\C$ or $\Ha$. 
Under the notation as in Proposition~\ref{first Pont}, 
for $G=U(p,q;\K)$, we have the following $U(q)$-equivariant vector bundle isomorphism:
\begin{align*}
K\times _{K_H}\frak{p}/\frak{p}_H\simeq U(q;\K)\times _{U(q_1;\K)\times U(q_2;\K)}\frak{p}/\frak{p}_H. 
\end{align*}
Here, $q=q_1+q_2$. 
See \cite[Sect. 7.1]{kobayashiOshima} for notation. 
\end{lemma}
Since the proof is easy, we omit it.

\begin{proof}[Proof of Proposition~\ref{first Pont} (outline)]
Let $E$ be the tautological vector bundle of the Grassmannian manifold 
$U(q;\K)/U(q_1;\K)\times U(q_2;\K)$, that is,
\[E:=U(q;\K)\times_{U(q_1;\K)\times U(q_2;\K)} \K^{q_1}.   \]
From Lemma~\ref{lemma:vector_bundle}, we have a natural isomorphism 
$K\times _{K_H}\frak{p}/\frak{p}_H \simeq E^{\oplus p}$, 
hence $p_1(K\times _{K_H}\frak{p}/\frak{p}_H)=p\cdot p_1(E)$
by Whitney's sum formula. 
Therefore, it suffices to verify whether $p_1(E)$ vanishes or not. 

We consider only the case $\K=\C$. The other cases $\K=\R$ or $\Ha$ are similar. 
Since $E\otimes_\R \C$ is isomorphic to a complex vector bundle $E\oplus \overline{E}$, 
the Pontrjagin class of $E$ is determined by the Chern class of $E$, in particular, 
\[ p_1(E)=c_1(E)^2 -c_2(E). \]
Define invariant polynomials $c_k\in S^k(\frak{u}(n)^*)^{U(n)}$ by 
\[\det\left(\lambda I_n -\frac{X}{2\pi\sqrt{-1}}\right)=\lambda^n +c_1(X)\lambda^{n-1} +\cdots +c_n(X) \]
for $X\in \frak{u}(n)$. 
The normalization is given in such a way that $c_k$ yields the $k$-th Chern class. 
Then one has the algebra isomorphisms 
$S(\frak{k}^*)^K\simeq \R[c_1,\cdots, c_{q_1+q_2}]$ and 
$S(\frak{k}_H^*)^{K_H}\simeq \R[c_1,\cdots, q_1]\otimes \R[c_1,\cdots, q_2]$, 
and the restriction $S(\frak{k}^*)^K\to S(\frak{k}_H^*)^{K_H}$
is given by 
\[c_k\mapsto \sum_{i+j=k}c_i\otimes c_j.  \]
Via Chern--Weil homomorphism $\omega$, one has
\[c_1(E)^2-2c_2(E)=\omega ( (c_1^2-2c_2)\otimes 1). \]
One can easily check that $(c_1^2-2c_2)\otimes 1\not \in \ker \omega$ if and only if $p_1\geq 2$ or $p_2\geq 2$ by using Fact~\ref{fact of ker omega}.
\end{proof}

\begin{proof}[Proof of Proposition~\ref{Pont of SO(2p,2q)/SU(p,q)} (outline)]
Let $E$ be an associated bundle on $K/K_H=(SO(2p)\times SO(2q))/(U(p)\times U(q))$ defined by the tensor product of the natural representations, that is, 
\[ E:=K\times _{K_H} \C^p\otimes \C^q. \]
Then we have a $K$-invariant vector bundle isomorphism $K\times_{K_H}\frak{p}/\frak{p}_H\simeq E$. 
In the same way with the proof of Proposition~\ref{first Pont}, 
it is enough to verify whether $p_1(E)$ vanishes or not. 

Define invariant polynomials $p_k\in S^k(\frak{o}(n)^*)^{O(n)}$ by 
\begin{align*}
\det \left(\lambda I_n+\frac{1}{2\pi}X\right)&=\lambda^n +q_1(X)\lambda^{n-1}+\cdots + q_n(X),\\
p_k&=q_{2k}
\end{align*}
for $X\in \frak{o}(n)$. 
Moreover, we define an invariant polynomial $e_n\in S^n(\frak{o}(2n)^*)^{SO(2n)}$ by 
\[ e_n(X)=\frac{1}{2^n n!}\left(\frac{1}{2\pi}\right)^n \sum_{\sigma\in \frak{S}_{2n}}(\operatorname{sgn} \sigma) X_{\sigma(1)\sigma(2)}\cdots X_{\sigma(2n-1)\sigma(2n)}. \]
We define invariant polynomials $c_k\in S^k(\frak{u}(n)^*)^{U(n)}$ as in the proof of Proposition~\ref{first Pont}. 
Then we have 
\begin{align*}
S(\frak{k}^*)^K&\simeq (\R[p_1,\cdots, p_p]\oplus e_p \R[p_1,\cdots, p_p])\otimes (\R[p_1,\cdots, p_q]\oplus e_p \R[p_1,\cdots, p_q]), \\
S(\frak{k}_H^*)^{K_H}&\simeq \R[c_1,\cdots, c_p]\otimes \R[c_1,\cdots, c_q]. 
\end{align*}
Via Chern--Weil homomorphism, one has 
\begin{align*}
p_1(E)&=c_1(E)^2-2c_2(E)\\
&=\omega( (c_1\otimes 1+1\otimes c_1)(c_1\otimes 1+1\otimes c_1)-2(c_2\otimes 1+1\otimes c_2))\\
&=\omega(p_1\otimes 1+1\otimes p_1+2 c_1\otimes c_1). 
\end{align*}
One can easily check that $ p_1\otimes 1+ 1\otimes p_1 +2c_1\otimes c_1\not \in \ker \omega$ if and only if $p\geq 2$ by using Fact~\ref{fact of ker omega}. 
\end{proof}

\section{Applications of Adams' theorem}\label{section:adams}
Kobayashi--Yoshino \cite{kobayoshi} found a surprising relationship between the existence problem of compact Clifford--Klein forms of tangential homogeneous spaces and Adams' results \cite{adams} on vector fields on spheres. 
This section pursues their idea. 
\begin{proposition}[Kobayashi--Yoshino {\cite[Prop.~5.5.1]{kobayoshi}}, Yoshino \cite{yoshino}]\label{application_of_adams}
Let $(G, H)$ and $(G, H^a)$ be symmetric pairs which are locally isomorphic to one of the following list and suppose that $G$ is connected. 
Then neither $G_\theta/H_\theta$ nor $G_\theta/H^a_\theta$ 
admit compact Clifford--Klein forms. 
\begin{itemize}
\item $(G,H,H^a)=(SO_0(p,q+1),SO_0(p,q),SO_0(p,1)\times SO(q))$ for $q\geq \rho(p,\R)$, where $\rho(p,\R)$ denotes the Hurwitz--Radon number (see Definition~\ref{Adams thm}), 
\item $(G,H=H^a)=(SU(p,2),S(U(p,1)\times U(1)))$ $(p$ is odd$)$, 
\item $(G,H,H^a)=(SO^*(2(2p)),SO^*(2(2p-1))\times SO^*(2), U(2p-1,1))$ $(p\geq 3)$,
\item $(G,H,H^a)=(SO(2p,\C), SO(2p-1,\C), SO_0(2p-1,1))$ $(1\leq p$ and $p\neq 1,2,4)$.
\end{itemize}
\end{proposition}

For the proof, 
see \cite{kobayoshi} for the first case and  \cite{yoshino} for the other cases. 

\begin{definition}[{\cite{adams,Hurwitz}}]\label{Adams thm}
For a positive integer $n$, 
when we write $n=2^k(2\ell+1)$, $k=4\alpha+\beta$ 
($k, \ell, \alpha, \beta\in \Z_{\geq 0}$, $0\leq \beta \leq 3$) uniquely, 
the Hurwitz--Radon number $\rho(n,\R)$ of $n$ is defined by 
\begin{align*}
\rho(n,\R)&:=8\alpha+2^\beta. 
\end{align*}
\end{definition}

\section*{Acknowledgments}
I would like to express my gratitude to my supervisor Taro Yoshino for constructive comments and warm encouragement. 
I would also like to thank Prof. Toshiyuki Kobayashi for giving me insightful advice. 
Moreover, I would also like to thank Yosuke Morita and Takayuki Okuda for valuable comments and discussions. 
I wish to thank the anonymous referee for the careful reading and helpful suggestions. 
This work was supported by the Program for Leading Graduate Schools, MEXT, Japan.

\end{document}